\documentclass[preprint,review,10pt]{amsart}
\usepackage{amsmath,amssymb,amsfonts,xspace}
\newtheorem{theorem}{Theorem}[section]

\theoremstyle{definition}
\newtheorem{definition}[theorem]{Definition}
\newtheorem{example}[theorem]{Example}

\newtheorem{corollary}[theorem]{Corollary}
\newtheorem{lem}[theorem]{Lemma}

\theoremstyle{remark}
\newtheorem{remark}[theorem]{Remark}

\numberwithin{equation}{section}

\begin{document}

\title{ (weakly) $(s,n)$-closed hyperideals }

\author{Mahdi Anbarloei}
\address{Department of Mathematics, Faculty of Sciences,
Imam Khomeini International University, Qazvin, Iran.
}

\email{m.anbarloei@sci.ikiu.ac.ir}





\begin{abstract}
A multiplicative hyperring is a well-known  type of algebraic hyperstructures which extend a ring  to  a structure in which the addition is an operation but the multiplication is a hyperoperation. Let $G$ be a commutative  multiplicative hyperring and $s,n \in \mathbb{Z}^+$. A proper hyperideal $Q$ of $G$ is called (weakly) $(s,n)$-closed  if ($0 \neq a^s \subseteq Q$ ) $a^s \subseteq Q$ for $a \in G$ implies $a^n \subseteq Q$.    In this paper, we aim to  investigate  (weakly) $(s,n)$-closed hyperideals and give some results explaining  the structures of these notions. 
\end{abstract}
\maketitle

\section{Introduction}
Hyperstructures, which are a generalization of classical algebraic structures, take an important place in both pure and applied mathematics. These structures was first introduced by Marty in 1934 \cite{marty}. He published some notes on hypergroups as a generalization of groups. Later on, many authors have worked on this new topic of modern algebra and developed it \cite{f1,f2,f3,f4,f5,f7,f8,f9}. Similar to hypergroups, hyperrings are algebraic structures, subsitutiting both or only one of the binary operations of addition and multiplication by hyperoperations.
A importan type of a hyperring called the multiplicative hyperring was introduced by Rota in 1982 \cite{f14}. In this hyperstructure,  the multiplication is a hyperoperation, while the addition is an operation.  Many illustrations and results of the multiplicative hyperring can be seen in  \cite{ameri5, ameri6, anb1, anb2, anb3, Ghiasvand, ul}. 

A hyperoperation "$\circ $" on nonempty set $A$ is a mapping of $A \times A$ into $P^*(A)$ where $P^*(A)$ is the family of all nonempty subsets of $A$. If "$\circ $" is a hyperoperation on $A$, then $(A,\circ)$ is called hypergroupoid. The hyperoperation on $A$ can be extended to  subsets of $A$. Let $A_1,A_2$ be two subsets of $A$ and $a \in A$, then $A_1 \circ A_2 =\cup_{a_1 \in A_1, a_2 \in A_2}a_1 \circ a_2,$ and $ A_1 \circ a=A_1 \circ \{a\}.$ A hypergroupoid $(A, \circ)$ is called  a semihypergroup if $\cup_{v \in b \circ c}a \circ v=\cup_{u \in a \circ b} u \circ c $ for all $a,b,c \in A$ which means $\circ$ is associative. A semihypergroup is said to be a hypergroup if  $a \circ A=A=A\circ a$ for all  $a \in A$. A nonempty subset $B$ of a semihypergroup $(A,\circ)$ is called a
subhypergroup if  we have $a \circ B=B=B \circ a$ for all $b \in B$. Recall from \cite{f10} that an algebraic structure $(G,+,\circ)$ is said to be commutative multiplicative hyperring if (1) $(G,+)$ is a commutative group; (2) $(G,\circ)$ is a semihypergroup; (3) for all $a, b, c \in G$, we have $a\circ (b+c) \subseteq a\circ b+a\circ c$ and $(b+c)\circ a \subseteq b\circ a+c\circ a$; (4) for all $a, b \in G$, we have $a\circ (-b) = (-a)\circ b = -(a\circ b)$; (5) for all $a,b \in G$, $a \circ b =b \circ a$
If in (3) the equality holds then we say that the multiplicative hyperring is strongly distributive. Let $(\mathbb{Z},+,\cdot)$ be the ring of integers. Corresponding to every subset $X \in P^\star(\mathbb{Z})$ with $\vert X\vert \geq 2$, there exists a multiplicative hyperring $(\mathbb{Z}_X,+,\circ)$ with $\mathbb{Z}_X=\mathbb{Z}$ and for any $a,b\in \mathbb{Z}_X$, $a \circ b =\{a.x.b\ \vert \ x \in X\}$.
A non empty subset $H$ of a multiplicative hyperring $R$ is a  hyperideal of $G$ if (i)  $a - b \in H$ for all $a, b \in H$; (ii) $r \circ a \subseteq H$ for all $a \in H $ and $r \in G$.
The notion of primeness of hyperideal in a multiplicative hyperring was conceptualized by Procesi and rota in \cite{f16}. Dasgupta extended the prime and primary hyperideals in multiplicative hyperrings in \cite{das}.
Recall from \cite{das} that a proper hyperideal $P$ of multiplicative hyperring  $G$ is called a prime hyperideal if $xa \circ b \subseteq P$ for $a,b \in G$ implies that $a \in P$ or $b \in P$. The intersection of all prime hyperideals of $R$ containing hyperideal $I$ is called the prime radical of $I$, being denoted by $rad (I)$. If the multiplicative hyperring $G$ does not have any prime hyperideal containing $I$, we define $rad(I)=G$. Let {\bf C} be the class of all finite products of elements of $R$ i.e. ${\bf C} = \{r_1 \circ r_2 \circ . . . \circ r_n \ : \ r_i \in R, n \in \mathbb{N}\} \subseteq P^{\ast }(R)$. A hyperideal $I$ of $R$ is said to be a {\bf C}-hyperideal of $R$ if, for any $A \in {\bf C}, A \cap I \neq \varnothing $ implies $A \subseteq I$.
Let I be a hyperideal of $R$. Then, $D \subseteq rad(I)$ where $D = \{r \in R: r^n \subseteq I \ for \ some \ n \in \mathbb{N}\}$. The equality holds when $I$ is a {\bf C}-hyperideal of $R$(\cite {das}, proposition 3.2). Recall that a hyperideal $I$ of $R$ is called a strong {\bf C}-hyperideal if for any $E \in \mathfrak{U}$, $E \cap I \neq \varnothing$, then $E \subseteq I$, where $\mathfrak{U}=\{\sum_{i=1}^n A_i \ : \ A_i \in {\bf C}, n \in \mathbb{N}\}$ and ${\bf C} = \{r_1 \circ r_2 \circ . . . \circ r_n \ : \ r_i \in R, n \in \mathbb{N}\}$ (for more details see \cite{phd}). Recall from \cite{ameri} that 
a proper hyperideal $I$ of multiplicative hyperring
 $R$ is maximal in R if for
any hyperideal $J$ of $R$ with $I \subseteq J \subseteq R$ then $J = R$. Also, we say that $R$ is a local multiplicative hyperring, if it has just one maximal hyperideal.   An element $e$ in $G$ is  an identity element if $a \in a\circ e$ for all $a \in G$. Moreover, an element $e$ in $G$ is  a scalar identity element if $a= a\circ e$ for all $a \in G$.

 In this paper, we introduce  the notions of $(s,n)$-closed hyperideals and weakly $(s,n)$-closed hyperideals of a commutative multiplicative hyperring.   Many specific results  are presented to explain the structures of these notions. Additionally,  we study the relationships between the new
hyperideals and  classical hyperideals and explore some ways to connect them. 
Throughout this paper, $G$ denotes a commutative multiplicative hyperring.
\section{ $(s, n)$-closed hyperideals }
In this section, we define the concept of $(s, n)$-closed hyperideals  and present some basic properties of $(s, n)$-closed hyperideals. 
\begin{definition}
Let $s,n \in \mathbb{Z}^+$. A proper hyperideal $Q$ of $G$  is called  $(s,n)$-closed if $a^s \subseteq Q$ for $a \in G$ implies $a^n \subseteq Q$.
\end{definition}
It is clear that every $(s,n)$-closed hyperideal of $G$ is $(s^{\prime},n^{\prime})$-closed for all $s,s^{\prime},n,n^{\prime} \in \mathbb{Z}^+$ with $s^{\prime} \leq s$ and $n^{\prime} \geq n$.
\begin{example}
Consider the set of all integers $\mathbb{Z}$ with ordinary addition  and the  hyperoperation $\circ$ defined as $\alpha \circ \beta =\{2\alpha \beta,  4\alpha \beta\}$ for all $\alpha, \beta \in \mathbb{Z}$. Then $105 \mathbb{Z}$ is an $(s,3)$-closed hyperideal of the multiplicative hyperring $(\mathbb{Z},+,\circ)$ for all $s \in \mathbb{Z}^+$. 
\end{example}
A proper hyperideal $I$ of $G$ refers to an $n$-absorbing hyperideal of $G$ if  $x_1 \circ \cdots \circ x_{n+1} \subseteq I$ for $x_1,\cdots,x_{n+1}   \in G$ implies that there exist  $n$ of the $x_i^,$s whose product is a subset of  $I$ \cite{anb4}.
\begin{theorem} \label{1}
Let $Q$ be an $n$-absorbing ${\bf C}$-hyperideal of $G$. Then $Q$ is an $(s,n)$-closed hyper for each $s \in \mathbb{Z}^+$.
\end{theorem}
\begin{proof}
Assume that $Q$ is an $n$-absorbing hyperideal of $G$ and $a^s=a^n \circ a^{s-n} \subseteq Q$ for $a \in G$ and $s,n \in \mathbb{Z}^+$ with $n <s$. Then we have $a^n \circ b \subseteq Q$ for every $b \in a^{s-n}$. Since $Q$ is an $n$-absorbing hyperideal of $G$, we get $a^n \in Q$ or $a^{n-1}\circ b \in Q$. In the first case, we are done. In the second case, we have $a^{s-1} \subseteq Q$. Consider $a^{s-1}=a^n \circ a^{s-n-1} \subseteq Q$ and pick $c \in a^{s-n-1}$. By continuing the process mentioned, we obtain $a^n \subseteq Q$. Hence $Q$ is $(s,n)$-closed for $s >n$. If $s \leq n$, then $Q$ is $(s,n)$-closed clearly. Thus $Q$ is an $(s,n)$-closed hyperideal for each $s \in \mathbb{Z}^+$.
\end{proof}
\begin{theorem}\label{2}
Let $Q_1,\cdots,Q_t$ be some prime hyperideals of $G$. Then $Q_1 \circ \cdots \circ Q_t$ is an $(s,n)$-closed hyperideal of $G$ for all $s,n \in \mathbb{Z}^+$ with $1 \leq s$ and $\min \{s,t\} \leq n$.
\end{theorem}
\begin{proof}
Assume that $a^s \subseteq Q_1 \circ \cdots \circ Q_t$ for $a \in G$. This means $a^s \subseteq Q_i$ for each $i \in \{1,\cdots,t\}$. Since $Q_i$ is a prime hyperideal of $G$, we have $a \in Q_i$. Hence $a^t \subseteq Q_1\circ \cdots \circ Q_t$. This implies that $a^n \subseteq Q_1\circ \cdots \circ Q_t$ for all $\min\{s,t\} \leq n$.
\end{proof}
\begin{theorem}\label{3}
Let $Q_1,\cdots,Q_t$ be some hyperideals of $G$ such that for each $i \in \{1,\cdots,t\}$,  $Q_i$ is $(s_i,n_i)$-closed with $s_i,n_i \in \mathbb{Z}^+$. Then
\end{theorem}
\begin{itemize}
\item[\rm(i)]~ $Q_1 \circ \cdots \circ Q_t$ is $(s,n)$-closed for $s,n \in \mathbb{Z}^+$ such that $s \leq \min\{s_1,\cdots,s_t\}$ and $n \geq \min \{s,n_1+\cdots+n_t\}.$
\item[\rm(ii)]~$Q_1 \cap \cdots \cap Q_t$ is $(s,n)$-closed for $s,n \in \mathbb{Z}^+$ such that  $s \leq \min\{s_1,\cdots,s_t\}$ and $n \geq \min \{s,\max\{n_1,\cdots,n_t\}\}.$
\end{itemize}
\begin{proof}
$(i)$ Assume that $a^s \subseteq Q_1 \circ \cdots \circ Q_t$ for  $s \leq \min\{s_1,\cdots,s_t\}$ and $a \in G$. This means $a^s \subseteq Q_i$ for all $i \in \{1,\cdots,t\}$ and so $a^{s_i} \subseteq Q_i$. Since  $Q_i$ is $(s_i,n_i)$-closed, we get $a^{n_i} \subseteq Q_i$. Then we conclude that $a^{n_1} \circ \cdots \circ a^{n_t}=a^{n_1+\cdots+n_t} \subseteq Q_1 \circ \cdots \circ Q_t$. This implies that $a^n \subseteq Q_1 \circ \cdots \circ Q_t$ for $n_1 \cdots +n_t \leq n$. Consequently, we have $a^n \subseteq Q_1 \circ \cdots \circ Q_t$ for all $n \geq \min \{s,n_1+\cdots+n_t\} $.\\
$(ii)$ Suppose that $a^s \subseteq Q_1 \cap \cdots Q_t$ for $s \leq \min\{s_1,\cdots,s_t\}$ and $a \in G$. Hence $a^s \subseteq Q_i$ for all $i \in \{1,\cdots,t\}$ and then $a^{s_i} \subseteq Q_i$. Since $Q_i$ is $(s_i,n_i)$-closed, we obtain $a^{n_i} \subseteq Q_i$. Thus we have $a^n \subseteq Q_i$ for every $i \in \{1,\cdots,t\}$ and $n \geq \max\{n_1,\cdots,n_t\}.$ Hence $a^n \subseteq Q_1 \cap \cdots \cap Q_t$ for all $n \geq \min \{s,\max\{n_1,\cdots,n_t\}\}$.
\end{proof}

\begin{corollary}\label{4}
Let $Q_1, \cdots,Q_t$ be some $(s,n)$-closed hyperideals of $G$ for some $s,n \in \mathbb{Z}^+$. Then $Q_1 \cap \cdots \cap Q_t$ is an $(s,n)$-closed hyperideal of $G$.
\end{corollary}
Recall from \cite{ameri} that two hyperideals $P$ and $Q$ of $G$ are coprime if $P+Q=G$.
\begin{corollary}\label{5}
Let $Q_1, \cdots,Q_t$ be some $(s,n)$-closed hyperideals of $G$ for some $s,n \in \mathbb{Z}^+$ such that $Q_1, \cdots, Q_t$ are pairwise coprime. Then $Q_1 \circ \cdots \circ Q_t$ is an $(s,n)$-closed hyperideal of $G$.
\end{corollary}
\begin{theorem}\label{6}
Let $P$ and $Q$ be two hyperideals of $G$ such that $Q$ is an $(s,2)$-closed strong {\bf C}-hyperideal for $s \in \mathbb{Z}^+$. If $P^s \subseteq Q$, then $P^2+P^2 \subseteq Q$.
\end{theorem}
\begin{proof}
Assume that $a, b \in P$. Since  $P^s \subseteq Q$, we have $a^s, b^s, (a+b)^s \subseteq Q$. Since $Q$ is an $(s,2)$-closed hyperideal of $G$, we get $a^2,b^2,(a+b)^2 \subseteq Q$. Pick $c \in (a+b)^2$.  Since $(a+b)^2 \subseteq a^2+a \circ b +a \circ b +b^2$, we have $c \in a^2+d+b^2$ for some $d \in a \circ b +a \circ b $ and so $c-d \in a^2+b^2  \subseteq Q$. Then we get $d \in Q$ as $c \in (a+b)^2 \subseteq Q$. Since $d \in a \circ b +a \circ b $ and $Q$ is a strong {\bf C}-hyperideal of $G$, we obtain $a \circ b +a \circ b \subseteq Q$. Hence $P^2+P^2 \subseteq Q$.
\end{proof}
We define the relation $\gamma$ on a multiplicative hyperring $(G, +, \circ)$ as follows:\\ 
$x \gamma y$ if and only if $\{x,y\} \subseteq A$ where $A$ is a finite sum of finite products of
elements of $G$. \\This means 
$x \gamma y $ if and only if there exist $ z_1, ... , z_n \in G$  that $\{x,y\} \subseteq \sum_{j \in J} \prod_{i \in I_j} z_i$ and $ I_j, J \subseteq \{1,... , n\}$.
 The transitive closure of $\gamma$ is denoted by $\gamma^{\ast}$. The relation $\gamma^{\ast}$ is the smallest equivalence relation on a multiplicative hyperring $G$ such that the
quotient $G/\gamma^{\ast}$, the set of all equivalence classes, is a fundamental ring. Let $\mathbb{A}$
be the set of all finite sums of products of elements of $G$. The
definition of $\gamma ^{\ast}$ can be rewritten on $G$, i.e.,
$x\gamma^{\ast}y$ if and only if  there exist  $ z_1, ... , z_n \in G$ with $z_1 = x, z_{n+1 }= y$ and $u_1, ... , u_n \in \mathbb{A}$ such that
$\{z_i, z_{i+1}\} \subseteq u_i$ for $i \in \{1, ... , n\}$.
Suppose that $\gamma^{\ast}(x)$ is the equivalence class containing $x \in G$. Define $\gamma ^{\ast}(x) \oplus \gamma^{\ast}(y)=\gamma ^{\ast}(z)$ for all $z \in \gamma^{\ast}(x) + \gamma ^{\ast}(y)$ and $\gamma ^{\ast}(x) \odot \gamma ^{\ast}(y)=\gamma ^{\ast}(w)$ for all $w \in  \gamma^{\ast}(x) \circ \gamma ^{\ast}(y)$. Then $(G/\gamma ^{\ast},+,\odot)$ is a ring, which is called a fundamental ring of $G$ \cite{sorc4}.

\begin{theorem}\label{7}
Let $Q$ be a hyperideal of $G$. Then  $Q$  is an $(s,n)$-closed hyperideal of $(G,+,\circ)$ if and only if $Q/\gamma ^{\ast}$ is  an $(s,n)$-closed ideal of $(G/\gamma ^{\ast},\oplus,\odot)$. 
\end{theorem}
\begin{proof}
($\Longrightarrow$) Let $\underbrace{ a \odot \cdots \odot a}_s \in Q/\gamma ^{\ast}$  for some $a \in G/\gamma ^{\ast}$. Therefore there exists $x\in G$ such
that $a =\gamma^{\ast}(x)$ and $\underbrace{a \odot \cdots \odot a}_s= \underbrace{\gamma^{\ast}(x) \odot \cdots \odot \gamma^{\ast}(x)}_s=\gamma^{\ast}(x^s)$. Since $
\gamma^{\ast}(x^s) \in Q/\gamma^{\ast}$, we obtain $x^s\subseteq Q$. Since $Q$ is an $(s,n)$-closed hyperideal of $G$, we conclude that  $x^n \subseteq Q$. Thus $\underbrace{a \odot \cdots \odot a}_n=\underbrace{ \gamma^{\ast}(x) \odot \cdots \odot \gamma^{\ast}(x)}_n=\gamma^{\ast}(x^n)\in Q/\gamma ^{\ast}$ which implies  $Q/\gamma ^{\ast}$ is an $(s,n)$-closed  ideal of $G/\gamma ^{\ast}$.\\
($\Longleftarrow$) Assume that $x^s \subseteq Q$ for some $x \in G$. Therefore $\gamma^{\ast}(x) \in G/\gamma^{\ast} $ and
$\underbrace{\gamma^{\ast}(x) \odot \cdots  \odot\gamma^{\ast}(x)}_s= \gamma^{\ast}(x^s) \in Q/\gamma^{\ast}$. Since $Q/\gamma ^{\ast}$ is an $(s,n)$-closed ideal of $G/\gamma ^{\ast}$,  we get $\underbrace{\gamma^{\ast}(x) \odot \cdots \odot \gamma^{\ast}(x)}_n= \gamma^{\ast}(x^n)\in Q/\gamma^{\ast}$ which  means $x^n \subseteq Q$. Consequently $Q$ is an $(s,n)$-closed  hyperideal of $G$.
\end{proof}

 Let us define $\mathfrak{C}(Q)=\{(s,n) \in \mathbb{Z}^+ \times \mathbb{Z}^+ \ \vert \ Q \text{ is (s,n)-closed }\}$ for some proper hyperideal $Q$ of $G$. Then we have $\{(s,n) \in \mathbb{Z}^+ \times \mathbb{Z}^+ \ \vert \ 1 \leq s \leq n \} \subseteq \mathfrak{C}(Q) \subseteq \mathbb{Z}^+ \times \mathbb{Z}^+$.  Moreover, $rad(Q)=Q$ if and only if $\mathfrak{C}(Q)=\mathbb{Z}^+ \times \mathbb{Z}^+$.

 \begin{theorem}\label{9}
 Let $Q$ be a proper hyperideal of $G$ and $s,n \in \mathbb{Z}^+$. If $(s,n), (s+1,n+1) \in \mathfrak{C}(Q)$ with $s \neq n$, then $(s+1,n) \in \mathfrak{C}(Q)$.
 \end{theorem}
 \begin{proof}
 Let $s <n$. Clearly, we have $(s+1,n) \in \mathfrak{C}(Q)$. Now, let us conseder $s>n$. Assume that $a^{s+1} \subseteq Q$ for some $a \in G$. Since $(s+1,n+1) \in \mathfrak{C}(Q)$, we get $a^{n+1} \subseteq Q $. Since $n+1 \leq s$, we obtain $a^s \subseteq Q$ which implies $a^n \subseteq Q$ as $(s,n) \in \mathfrak{C}(Q)$. Hence $(s+1,n) \in \mathfrak{C}(Q)$.
 \end{proof}
  \begin{lem} \label{10}
 Let $Q$ be a proper hyperideal of $G$ and $s,n \in \mathbb{Z}^+$. If $(s,n) \in \mathfrak{C}(Q)$, then $(s^{\prime},n^{\prime}) \in \mathfrak{C}(Q)$ for $s^{\prime},n^{\prime} \in \mathbb{Z}^+$ such that $1 \leq s^{\prime} \leq s$ and $n \leq n^{\prime}$.
 \end{lem}
 \begin{proof}
Straightforward.
 \end{proof}
 \begin{theorem}\label{11}
  Let $Q$ be a proper ${\bf C}-$hyperideal of $G$.
\begin{itemize}
\item[\rm(i)]~ If $(n,2), (n+1,2) \in \mathfrak{C}(Q)$ for $n \in \mathbb{Z}^+$ such that $n  \geq 3$, then $(n+2,2) \in \mathfrak{C}(Q)$ and then $(t,2) \in \mathfrak{C}(Q)$ for all $t \in \mathbb{Z}^+$.
\item[\rm(ii)]~ If $(s,n) \in \mathfrak{C}(Q)$ for $s,n \in \mathbb{Z}^+$ such that $n \leq \frac{s}{2}$, then $(s+1,n) \in \mathfrak{C}(Q)$ and then $(t,n) \in \mathfrak{C}(Q)$ for all $t \in \mathbb{Z}^+$.
\end{itemize}
 \end{theorem}
 \begin{proof}
 (i) Let $a^{n+2} \subseteq Q$ for some $a \in G$. Since $2n \geq n+2$, we have $(a^2)^n=a^{2n} \subseteq Q$. Pick $x \in a^2$.  Since $x^n \subseteq Q$ and $(n,2) \in \mathfrak{C}(Q)$, we get $x^2 \subseteq Q$. From $x \in a^2$ it follows that $x^2 \subseteq a^4$. Since $Q$ is a  ${\bf C}-$hyperideal of $G$, we conclude that $a^4 \subseteq Q$. Since $(n+1,2) \in \mathfrak{C}(Q)$ and $n \geq 3$, we obtain $a^2 \subseteq Q$. Consequently, $(n+2,2) \in \mathfrak{C}(Q)$. By a similar argument, we conclude that $(t,2) \in \mathfrak{C}(Q)$ for all $t \in \mathbb{Z}^+$ such that $t \geq n+3$. Thus, by Lemma \ref{10}, $(t,2) \in \mathfrak{C}(Q)$ for all $t \in \mathbb{Z}^+$.
 
 (ii) Assume that $a^{s+1} \subseteq Q$ for some $a \in G$. Therefore we have $(a^2)^s=a^{2s} \subseteq Q$. Choose $x \in a^2$. So $x^s \subseteq Q$. Since $(s,n) \in \mathfrak{C}(Q)$, we get $x^n \subseteq Q$. From  $x \in a^2$ it follows that $x^n \subseteq a^{2n}$. Then we have $a^{2n} \subseteq Q$ as $Q$ is a  ${\bf C}-$hyperideal of $G$. By the hypothesis, we have  $a^s \subseteq Q$. Since  $(s,n) \in \mathfrak{C}(Q)$,  we get $a^n \subseteq Q$ which implies $(s+1,n) \in \mathfrak{C}(Q)$. By a similar argument, we have $(t,n) \in \mathfrak{C}(Q)$ for all $t \in \mathbb{Z}^+$ such that $t \geq n$. Now, by Lemma \ref{10}, we conclude that $(t,n) \in \mathfrak{C}(Q)$ for all $t \in \mathbb{Z}^+$.
 \end{proof}
 Assume that $Q$ is a proper hyperideal of $G$ and $s,n \in \mathbb{Z}^+$. We consider the mappings $\omega: \mathbb{Z}^+ \longmapsto \mathbb{Z}^+$, defined by $\omega_Q(s)=\min \{n \ \vert \  Q \text{ is (s,n)-closed}\} \in \{1,\cdots,s\}$,  and $\Omega: \mathbb{Z}^+ \longmapsto \mathbb{Z}^+ \cup \{\infty\}$, defined by $\Omega_Q(n)=\sup \{s \ \vert \  Q \text{ is (s,n)-closed}\} \in \{n,n+1,\cdots\} \cup \{\infty\}$. The rows of $\mathfrak{C}(Q)$ determine $\omega_Q$ and the columns of $\mathfrak{C}(Q)$ determine $\Omega_Q$. Since $(n,n) \in \mathfrak{C}(Q)$ for all $n \in \mathbb{Z}^+$, we have $1 \leq \omega_Q(s) \leq s$. Moreover, by Lemma \ref{10}, we conclude  that $\omega_Q(s) \leq \omega_Q(s+1)$ and $\Omega_Q(n) \leq \Omega_Q(n+1)$.
 \begin{remark} \label{0}
 Let $P$ and $Q$ be two proper hyperideals of $G$. It can be easily seen that $\mathfrak{C}(P) \subseteq \mathfrak{C}(Q)$ for  if and only if $\omega_P(s) \leq \omega_Q(s)$ for all $s \in  \mathbb{Z}^+$ if and only if $\Omega_P(n) \leq  \Omega_Q(n)$ for all $n \in  \mathbb{Z}^+$.
  \end{remark} 
 \begin{theorem} \label{12}
 Let $Q$ be a proper hyperideal of $G$ and $s,n \in \mathbb{Z}^+$. If $\omega_Q(s) <s$, then either $\omega_Q(s+1)=\omega_Q(s)$ or $\omega_Q(s+1) \geq \omega_Q(s)+2$.
 \end{theorem}
 \begin{proof}
 Let $\omega_Q(s+1) = \omega_Q(s)+1$ and $\omega_Q(s)=n$. Then $n<s$ and $\omega_Q(s+1) =n+1$. Hence we have $(s+1,n+1), (s,n) \in \mathfrak{C}(Q)$. By Theorem \ref{9}, we conclude that $(s+1,n) \in \mathfrak{C}(Q)$ which implies $\omega_Q(s+1) \leq n$ which is a contradiction.
 \end{proof}
 \begin{theorem} \label{13}
 Let $Q$ be a proper hyperideal of $G$ and $s,n \in \mathbb{Z}^+$. If $\Omega_Q(n) >n$, then either $\Omega_Q(n+1)=\Omega_Q(n)$ or $\Omega_Q(n+1) \geq \Omega_Q(n)+2$.
 \end{theorem}
 \begin{proof}
 By using an argument similar to that in the proof of Theorem \ref{12},
one can easily complete the proof.
 \end{proof}
 \begin{theorem}\label{14}
 Let $P$ and $Q$ be  proper hyperideals of $G$. Then 
 \begin{itemize}
\item[\rm(i)]~$\omega_{P\cap Q} \leq \omega_P \vee \omega_Q$.
\item[\rm(ii)]~$\Omega_P \wedge \Omega_Q \leq \Omega_{P\cap Q}$.
 \end{itemize}
 \end{theorem}
 \begin{proof}
 Put $ \omega_P(s)=n^{\prime}$ and  $\omega_Q(s)=n^{\prime \prime}$ for  $s \in \mathbb{Z}^+$ such that $n=\max \{n^{\prime}, n^{\prime \prime}\}.$ By Lemma \ref{10} and Corollary \ref{4}, we have $(s,n) \in \mathfrak{C}(P) \cap \mathfrak{C}(Q) \subseteq \mathfrak{C}(P \cap Q)$.  Therefore we conclude that $\omega_{P\cap Q}(s) \leq n$. Since $n = (\omega_P \vee \omega_Q)(s)$, we have $\omega_{P\cap Q}(s) \leq (\omega_P \vee \omega_Q)(s)$.
 
 (ii) This can be proved in a very similar manner to the way in
which (i) was proved.
 \end{proof}
 \begin{theorem} \label{15}
 Let $P$ and $Q$ be  proper hyperideals of $G$. Then $\omega_{P\cap Q} = \omega_P \vee \omega_Q$ if and only if $\mathfrak{C}(P) \cap \mathfrak{C}(Q) = \mathfrak{C}(P \cap Q)$.
 \end{theorem}
 \begin{proof}
 $(\Longrightarrow) $Let $\omega_{P\cap Q} = \omega_P \vee \omega_Q$. This means that $\omega_P \leq \omega_{P\cap Q}$ and $\omega_Q \leq \omega_{P\cap Q}$. By Remark \ref{0}, we get $\mathfrak{C}(P \cap Q) \subseteq \mathfrak{C}(P) \cap \mathfrak{C}(Q)$. On the other hand, we have $\mathfrak{C}(P) \cap \mathfrak{C}(Q) \subseteq  \mathfrak{C}(P \cap Q)$ by Corollary \ref{4}. Consequently, $\mathfrak{C}(P) \cap \mathfrak{C}(Q) = \mathfrak{C}(P \cap Q)$.
 
 $(\Longleftarrow)$ Let $\mathfrak{C}(P) \cap \mathfrak{C}(Q) = \mathfrak{C}(P \cap Q)$. This implies that $\omega_P \leq \omega_{P\cap Q}$ and $\omega_Q \leq \omega_{P\cap Q}$ by Remark \ref{0}. Therefore we obtain $\omega_P \vee \omega_Q \leq \omega_{P\cap Q}$.  On the other hand, by Theorem \ref{14}, we get $\omega_{P\cap Q} \leq  \omega_P \vee \omega_Q$. Consequently, $\omega_{P\cap Q} = \omega_P \vee \omega_Q$.
 \end{proof}
 \begin{theorem} \label{16}
 Let $P$ and $Q$ be  proper hyperideals of $G$. Then $\Omega_{P\cap Q} = \Omega_P \wedge \Omega_Q$ if and only if $\mathfrak{C}(P) \cap \mathfrak{C}(Q) = \mathfrak{C}(P \cap Q)$.
 \end{theorem}
 \begin{proof}
 By an argument similar to that one given in  Theorem \ref{15}, we are done.
 \end{proof}
 In view of Theorem \ref{15} and Theorem \ref{16},  the following result is obtained.

 \begin{corollary} \label{17}
 Let $P$ and $Q$ be  proper hyperideals of $G$. Then $\omega_{P\cap Q} = \omega_P \vee \omega_Q$ in and only if $\Omega_{P\cap Q} = \Omega_P \wedge \Omega_Q$
 \end{corollary}
 \section{weakly $(s,n)$-closed hyperideals}
 \begin{definition}
 Let $Q$ be a proper hyperideal of $G$ and $s,n \in \mathbb{Z}^+$. $Q$ refers to a weakly $(s,n)$-closed hyperideal if  $0 \notin x^s \subseteq Q$ for $x \in G$ implies  that $x^n \subseteq Q$. 
 \end{definition}
 \begin{example}
 Consider the multiplicative hyperring $(\mathbb{Z}_A,+,\circ)$ where $\mathbb{Z}_A=\mathbb{Z}$, $A=\{7,11\}$, $a \circ b =\{a \cdot x \cdot b \ \vert \ x \in A\}$ for all $a, b \in \mathbb{Z}_A$ and $+$ is ordinary addition. Then $\langle 390 \rangle$ is an weakly $(s,4)$-closed hyperideal of $\mathbb{Z}$ for all $s \in \mathbb{Z}^+$.
 \end{example}
 It can be easily verified that an intersection of weakly $(s,n)$-closed hyperideals of $G$ is weakly $(s,n)$-closed. Furthermore, every weakly $(s,n)$-closed hyperideal of $G$ is weakly $(s,n^{\prime})$-closed for all $n^{\prime} \in \mathbb{Z}^+$ such that $n^{\prime} \geq n$. Let $Q$ be a weakly $(s,n)$-closed ${\bf C}$-hyperideal of $G$. We define $x \in G$ to be an $(s,n)$-tough-zero element of $Q$ if $0 \in x^s$ and $x^n \nsubseteq Q$. Indeed, it means that $Q$ is not $(s,t)$-closed if and only if $Q$ has an $(s,n)$-tough-zero element.
 \begin{theorem} \label{18}
 Let $Q$ be a weakly $(s,n)$-closed strong ${\bf C}$-hyperideal of $G$. If $x$ is an $(s,n)$-tough-zero element of $Q$, then $0 \in (x+a)^s$ for all $a \in Q$. 
 \end{theorem}
 \begin{proof}
 Assume that $a \in Q$. So $\sum_{i=1}^s \tbinom{s}{i}x^{s-i}a^i  \subseteq Q$. On the other hand, $0+\sum_{i=1}^s \tbinom{s}{i}x^{s-i} a^i \subseteq x^s+\sum_{i=1}^s \tbinom{s}{i}x^{m-i}a^i $. Since $Q$ is a strong ${\bf C}$-hyperideal of $G$, we have $x^s+\sum_{i=1}^s \tbinom{s}{i}x^{s-i}a^i \subseteq Q$ which implies $(x+a)^s \subseteq  Q$ as $(x+a)^s \subseteq x^s+\sum_{i=1}^s \tbinom{s}{i}x^{m-i}a^i $. If $(x+a)^n \subseteq Q$, then $x^n+\sum_{i=1}^n \tbinom{n}{i}x^{n-i}a^i \subseteq Q$ because $Q$ is a strong ${\bf C}$-hyperideal of $G$. Since $\sum_{i=1}^n \tbinom{n}{i}x^{n-i}a^i \subseteq Q$, we get $x^n \subseteq Q$, a contradiction. Then $(x+a)^n \nsubseteq Q$. Since $Q$ is a weakly $(s,n)$-closed of $G$, we obtain  $0 \in (x+a)^s$. 
 \end{proof}
 Recall from \cite{ameri} that an element $x \in G$ is nilpotent if there exists an integer $t$ such that $0 \in x^t$. The set of all nilpotent elements of $G$ is denoted by $\Upsilon$. 
 \begin{theorem} \label{19}
 If $Q$ is a weakly $(s,n)$-closed strong ${\bf C}$-hyperideal of $G$ such that is not $(s,t)$-closed, then $Q \subseteq \Upsilon$.
 \end{theorem}
 \begin{proof}
 Suppose that $Q$ is a weakly $(s,n)$-closed strong ${\bf C}$-hyperideal of $G$. If $Q$ is not $(s,t)$-closed, then there exists $x \in G$ such that $x$ is  an $(s,n)$-tough-zero element of $Q$. Assume that $a \in Q$. We have $0 \in x^s$ which implies $0 \in (x+a)^s$ by theorem \ref{18}. This means that $x, x+a \in \Upsilon$. Since by Theorem 3.5 in \cite{ameri} $\Upsilon$ is a hyperideal of $G$ , we conclude that $a=(x+a)- x \in \Upsilon$. This means that $Q \subseteq \Upsilon$.
 \end{proof}
 \begin{theorem}\label{20}
Let $G$ be a strongly distributive multiplicative hyperring with scalar identity $e$ such that has an $i$-set and $s,n \in \mathbb{Z}^+$ with $s >n$. Then every hyperideal of $G$ contained in  $\Upsilon$ is weakly $(s,n)$-closed if and only if $0 \in x^s$ for all $x \in \Upsilon$.
\end{theorem}
\begin{proof}
$(\Longrightarrow)$ Let $0 \notin x^s$ for some $x \in \Upsilon$. Put $Q=\langle x^s \rangle$. So  $\langle x^s \rangle \subseteq \Upsilon$.  Since every hyperideal of $G$ contained in  $\Upsilon$ is weakly $(s,n)$-closed, $Q$ is a weakly $(s,n)$-closed hyperideal of $G$. Therefore $0 \notin x^s \subseteq Q$ which implies $x^n \subseteq Q$ and $0 \notin x^n$. Then for every $a \in x^n$ there exist $a_1,\cdots,a_t \in G$ and $x_1,\cdots,x_t \in x^s$ such that $a \in \sum_{i=1}^t a_i \circ x_i$. Hence $0 \in \sum_{i=1}^t x_i \circ a_i -a \subseteq \sum_{i=1}^t x^s \circ a_i -x^n=x^n \circ (\sum_{i=1}^t x^{s-n} \circ a_i-e)$. 
This means that $0 \in x^n \circ c-e$ for some $c \in \sum_{i=1}^t x^{s-n} \circ a_i$. Since $x \in \Upsilon$, we have $\sum_{i=1}^t x^{s-n} \circ a_i \in \Upsilon$. Hence   $c-e$ is invertible by Theorem 3.20 in \cite{ameri2}. Thus we have $0 \in x^n$ which is a contradiction.

$\Longleftarrow$  Let $0 \in x^s $ for all $x \in \Upsilon$. Suppose that $Q$ is a hyperideal of $G$ such that $Q \subseteq \Upsilon$ and $0 \notin a^s \subseteq Q$ for some $a \in G$ but $a^n \nsubseteq Q$. Then there exist $c \in a^n$ such that $c \notin Q$. If $c \in \Upsilon $, then $0 \in c^s \subseteq a^{sn}$ and so $a \in \Upsilon$. Hence $0 \in a^s$, a contradiction. If $c \notin \Upsilon$, then $0 \notin c^s$. Since $c^s \subseteq a^{ns} \subseteq Q$, we have $c^s \subseteq \Upsilon$. Let $d \in c^s$. Then $0 \in d^s \subseteq c^{s^2}$ which means $c \in \Upsilon$, a contradiction.
\end{proof}
 \begin{definition}
 Let $s,n \in \mathbb{Z}^+$. An element $a \in G$ is  $(s,n)$-regular if $a^n \subseteq a^s \circ b$ for some $b \in G$. An element $a \in G$ is  $(s,n)$-Regular if $a^n \subseteq a^s \circ B$ for some subset $B$ of $G$.
 
 \end{definition}
Recall from \cite{ameri2} that an element $a \in G$ is called weak zero divisor if for $0 \notin b \in G$ we have $0 \in a \circ b$. Denote the set of weak zero divisor by $Z_w(G)$. Also, recall from \cite{ameri} that an element $x \in G$ is said to be invertible if there exists $y \in G$, such that $e \in x\circ y$. The set of all invertible elements in $G$ is denoted by $U(G)$.
\begin{theorem}\label{21}
Assume that $G$ is a strongly distributive multiplicative hyperring with scalar identity $e$,   $a \in G \backslash ( Z_w(G) \cup U(G))$ and $s,n \in \mathbb{Z}^+$. Then $a$ is $(s,n)$-regular if and only if $s \leq n$.
\end{theorem}
\begin{proof}
 Let $a$ be $(s,n)$-regular and $s>n$. From $a^n \subseteq a^s \circ b$ for some $b \in G$, it follows that $0 \in a^s \circ b -a^n=a^n \circ (a^{s-n}\circ b -e)$. By the hypothesis, we have $0 \in a^{s-n}\circ b -e$ which implies $e \in a^{s-n} \circ b$. This means $e \in a \circ d$ for some $d \in a^{s-n-1} \circ b$. Then $a \in U(G)$ which is a contradiction. The converse is clear.
\end{proof}
\begin{theorem}\label{22}
Let $s,n \in \mathbb{Z}^+$ with $s >n$. If $a$ is $(s,n)$-regular, then $a$ is $(s+1,n)$-Regular.
\end{theorem}
\begin{proof}
Assume that $a$ is $(s,t)$-regular such that $s >n$. From $a^n \subseteq a^s \circ b$ for some $b \in G$, it follows that $a^n \subseteq a^s \circ b=a^n \circ (a^{s-n} \circ b) \subseteq (a^s \circ b) \circ  (a^{s-n} \circ b)=a^{s+1} \circ (a^{s-n-1} \circ b^2)$. Put $B=a^{s-n-1} \circ b^2$. Thus we have $a^n \subseteq a^{s+1} \circ B$ which means $a$ is $(s+1,n)$-Regular.
\end{proof}
\begin{theorem}\label{23}
If $a \in U(G)$, then $a$ is $(s,n)$-Regular for all $s,n \in \mathbb{Z}^+$.
\end{theorem}
\begin{proof}
Let $a \in U(G)$. Then there exists $b \in G$ such that $e \in a \circ b$. So we have $e \in a^{s-n} \circ b^{s-n}$. Therefore $a^n \circ e \subseteq a^s \circ b^{s-n}$.  If $x \in a^n$,  then $x \in x \circ e \subseteq a^n \circ e$ which implies $a^n \subseteq a^n \circ e$. Hence we get $a^n \subseteq a^s \circ b^{s-n}$. Put $B=b^{s-n}$. Then $a^n \subseteq a^s \circ B$ which means $a$ is $(s,n)$-Regular.
\end{proof}
A non-empty finite subset $\xi=\{e_1,\cdots,e_n\}$ of a multiplicative hyperring $G$ is said to be $i$-set if  $e_i \neq 0$ for at least one $1 \leq i \leq n$ and for every $x \in G$, $x \in \sum_{i=1}^n x \circ e_i$ \cite{Sen2}. 
\begin{theorem}\label{24}
Let $G$ be a strongly distributive multiplicative hyperring 
that has an $i$-set and $s,n \in \mathbb{Z}^+$ with $s >n$. Then every proper hyperideal of $G$ is weakly $(s,n)$-closed if and only if every non-nilpotent element of $G$ is $(s,n)$-Regular and $0 \in a^s$ for all $a \in \Upsilon$.
\end{theorem}
\begin{proof}
$(\Longrightarrow)$ Since every hyperideal of $G$ contained in  $\Upsilon$ is weakly $(s,n)$-closed, we have  $0 \in a^s$ for all $a \in \Upsilon$ by Theorem \ref{20}. Let $a \in G \backslash \Upsilon.$ Assume that $a \in U(G)$. Therefore $a$ is $(s,n)$-Regular by Theorem \ref{23}. Now, assume that $a \notin U(R)$. Put $Q=\langle a^s \rangle$. By the hypothesis, $Q$ is a weakly $(s,n)$-closed hyperideal of $G$. Therefore $0 \notin a^s \subseteq Q$ which implies $a^n \subseteq Q$. Then for every $x \in a^n$ there exist $b_{1_x},\cdots,b_{t_a} \in a^s$ and $c_{1_x},\cdots,c_{t_x} \in G$   such that $x \in \sum_{i=1}^t b_{i_x} \circ c_{i_x} \subseteq \sum_{i=1}^t a^s \circ c_{i_x}=a^s \circ \sum_{i=1}^t c_{i_x}$. Put $c_x =\sum_{i=1}^t c_{i_x}$. Then we obtain $a^n=\cup_{x \in a^n} \{x\} \subseteq \cup_{x \in a^n} a^s \circ c_x=a^s \circ \cup_{x \in a^n} c_x$. Now, put $B=\cup_{x \in a^n} c_x$. Thus, $a^n \subseteq a^s \circ B$ which means $a$ is $(s,n)$-Regular.

$(\Longleftarrow)$ Assume that $Q$ is a proper hyperideal of $G$ and $0 \notin a^s \subseteq Q$ for some $a \in G$. So $a \notin \Upsilon$. By the hypothesis, $a$ is $(s,n)$-Regular. This means $a^n \subseteq a^s \circ B$ for some subset $B$ of $G$. Since $a^s \circ B \subseteq Q$, we have $a^n  \subseteq Q$. Consequently, $Q$ is weakly $(s,n)$-closed.
\end{proof} 
Assume that $(G_1, +_1, \circ _1)$ and $(G_2, +_2, \circ_2)$ are two multiplicative hyperrings. Recall from \cite{f10} that a mapping $\psi$ from
$G_1$ into $G_2$ is said to be a hyperring good homomorphism if for all $x,y \in G_1$, $\psi(x +_1 y) =\psi(x)+_2 \psi(y)$ and $\psi(x\circ_1y) = \psi(x)\circ_2 \psi(y)$.

\begin{theorem} \label{homo} 
Let $G_1$ and $G_2$ be two multiplicative hyperrins, $\psi: G_1 \longrightarrow G_2$ a hyperring
good homomorphism and $s,n \in \mathbb{Z}^+$. 
\begin{itemize}
\item[\rm{(1)}]~ If $Q_2$ is a weakly $(s,n)$-closed hyperideal of $G_2$ and $\psi$ is injective, then $\psi^{-1}(Q_2)$ is a weakly $(s,n)$-closed hyperideal of $G_1$.
\item[\rm{(2)}]~ If $\psi$ is surjective and $Q_1$ is a weakly $(s,n)$-closed hyperideal of $G_1$ with $Ker (\psi) \subseteq Q_1$ , then $\psi(Q_1)$ is a weakly $(s,n)$-closed hyperideal of $G_2$.
\end{itemize}
\end{theorem}
\begin{proof}
(1) Let $0 \notin a^s \subseteq \psi^{-1}(Q_2)$ for some $a \in G_1$. Then we get $0 \notin \psi(a^s)=\psi(a)^s \subseteq Q_2$ as $\psi$ is injective. Since $Q_2$ is a weakly $(s,n)$-closed hyperideal of $G_2$, we obtain $(\psi(a))^n \subseteq Q_2$  which implies $\psi(a^n) \subseteq Q_2$ which means $a^n \subseteq \psi^{-1}(Q_2)$. Thus  $\psi^{-1}(Q_2)$ is a weakly $(s,n)$-closed hyperideal of $G_1$.

(2)  Let $0 \notin b^s \subseteq \psi(Q_1)$ for some $b \in G_2$. Then  $\psi(a)=b$ for some $a \in G_1$ as $\psi$ is surjective. Therefore $\psi(a^s)=\psi(a)^s \subseteq \psi(Q_1)$. Now, pick any $x \in a^s$. Then $\psi(x) \in \psi(a^s) \subseteq \psi(Q_1)$ and so there exists $y \in Q_1$ such that $\psi(x)=\psi(y)$. Then we have $\psi(x-y)=0$ which means $x-y \in Ker (\psi)\subseteq Q_1$ and so  $x \in Q_1$. So $a^s  \subseteq Q_1$. Since $Q_1$ is weakly $(s,n)$-closed and $0 \notin a^s$,   we get $a^n \subseteq Q_1$. This means that $\psi(a^n)=b^n \subseteq \psi(Q_1)$.  Consequently,  $\psi(Q_1)$  is a weakly $(s,n)$-closed hyperideal of $G_2$.
\end{proof}
\begin{corollary}
Let $P$ and $Q$ be two hyperideals of $G$ with $P \subseteq Q$ and $s,n \in \mathbb{Z}^+$. If $Q$ is a weakly $(s,n)$-closed hyperideal  of $G$, then  $Q/P$ is a weakly $(s,n)$-closed hyperideal of $G/P$.
\end{corollary}
\begin{proof}
By Theorem \ref{homo} and using the mapping $\pi:  G \longrightarrow G/P$ defined by $\pi(a)=a+P$ we are done. 
\end{proof}

Let $(G_1,+_1,\circ_1)$ and $(G_2,+_2,\circ_2)$ be two multiplicative hyperrings with non zero identity.  The set $G_1 \times G_2$  with the operation $+$ and the hyperoperation $\circ$  defined  as

$(x_1,x_2)+(y_1,y_2)=(x_1+_1y_1,x_2+_2y_2)$

$(x_1,x_2) \circ (y_1,y_2)=\{(x,y) \in R_1 \times R_2 \ \vert \ x \in x_1 \circ_1 y_1, y \in x_2 \circ_2 y_2\}$ \\
is a multiplicative hyperring \cite{ul}. Now, we give some characterizations of weakly $(s,n)$-closed hyperideals on cartesian product of commutative multiplicative hyperring.

\begin{theorem} \label{25}
Let $(G_1, +_1,\circ _1)$ and $(G_2,+_2,\circ_2)$ be two multiplicative hyperrings with  scalar identities $e_1$ and $ e_2$ respectively,  $Q_1$ a ${\bf C}$-hyperideal of $G_1$ and $s,n \in \mathbb{Z}^+$ . Then the followings are equivalent.
 \begin{itemize}
\item[\rm(i)]~  $Q_1 \times G_2$ is a weakly $(s,n)$-closed hyperideal of $G_1 \times G_2$. 
\item[\rm(ii)]~ $Q_1 $ is an $(s,n)$-closed hyperideal of $G_1 $.
\item[\rm(iii)]~ $Q_1 \times G_2$ is an $(s,n)$-closed hyperideal of $G_1 \times G_2$.
 \end{itemize}
\end{theorem}
\begin{proof}
(i) $\Longrightarrow$ (ii) Assume that $Q_1 \times G_2$ is a weakly $(s,n)$-closed hyperideal of $G_1 \times G_2$. By Theorem \ref{homo} (2), we conclude that $Q_1$  is a weakly $(s,n)$-closed hyperideal of $G_1$. Suppose that $Q_1 $ is not an $(s,n)$-closed hyperideal of $G_1 $. This implies that $Q_1$ has an $(s,n)$-tough-zero element $x$. Hence $0 \in x^s$ and $x^n \nsubseteq Q_1$. This implies that  $(0,0) \notin (x,e_2)^s \subseteq Q_1 \times G_2$ and $(x,e_2)^n \nsubseteq Q_1 \times G_2$ which is  a contradiction since $Q_1 \times G_2$ is a weakly $(s,n)$-closed hyperideal of $G_1 \times G_2$. Thus $Q_1 $ is an $(s,n)$-closed hyperideal of $G_1 $.

(ii) $\Longrightarrow$ (iii) The claim follows by Theorem  2.12 in \cite{Anderson} and Theorem \ref{7}.

(iii) $\Longrightarrow$ (i) This follows directly from the deﬁnitions.
\end{proof}
\begin{lem}\label{26}
Let $(G_1, +_1,\circ _1)$ and $(G_2,+_2,\circ_2)$ be two multiplicative hyperrings and $I_1$ and $I_2$ be hyperideals of $G_1$ and $G_2$, respectively. Then $I_1$ and $I_2$ are ${\bf C}$-hyperideals if and only if $I_1 \times I_2$ is ${\bf C}$-hyperideal of $G_1 \times G_2$.
\end{lem}
\begin{proof}
$\Longrightarrow$ Let $I_1$ and $I_2$ be ${\bf C}$-hyperideals of $G_1$ and $G_2$, respectively and $(a_1,b_1) \circ \cdots \circ (a_n,b_n) \cap I_1 \times I_2 \neq \varnothing$ for some $a_1,\cdots,a_n \in G_1$ and $b_1,\cdots,b_n \in G_2$. This means $(a,b) \in (a_1,b_1) \circ \cdots \circ (a_n,b_n)$ for some $(a,b) \in I_1 \times I_2$. Therefore we have $a \in a_1 \circ_1 \cdots \circ_1 a_n$ and $b \in b_1 \circ_2 \cdots \circ_2 b_n$. Since $I_1$ and $I_2$ are ${\bf C}$-hyperideals, we get $a_1 \circ_1 \cdots \circ_1 a_n \subseteq I_1$ and $b_1 \circ_2 \cdots \circ_2 b_n \subseteq I_2$. This implies that $(a_1,b_1) \circ \cdots \circ (a_n,b_n) \subseteq I_1 \times I_2$, as needed. 

$\Longleftarrow$ Let $a_1 \circ_1 \cdots \circ_1 a_n \cap I_1 \neq \varnothing$ for some $a_1, \cdots,a_n \in G_1$ and $b_1 \circ_2 \cdots \circ_2 b_n \cap I_2 \neq \varnothing$ for some  $b_1,\cdots,b_n \in G_2$. This means $(a_1,b_1) \circ \cdots \circ (a_n,b_n) \cap I_1 \times I_2 \neq \varnothing$. Since $I_1 \times I_2$ is ${\bf C}$-hyperideal of $G_1 \times G_2$,  we have $(a_1,b_1) \circ \cdots \circ (a_n,b_n) \subseteq  I_1 \times I_2$ which means  $a_1 \circ_1 \cdots \circ_1 a_n \subseteq I_1$ and $b_1 \circ_2 \cdots \circ_2 b_n \subseteq I_2$, as claimed.
\end{proof}
\begin{theorem}\label{27}
Let $(G_1, +_1,\circ 1)$ and $(G_2,+_2,\circ_2)$ be two multiplicative hyperrings with scalar  identities $e_1$ and $e_2$, repectively,  and $s,n \in \mathbb{Z}^+$. Then the followings are equivalent.
\begin{itemize}
\item[\rm{(i)}]~ $Q$ is a weakly $(s,n)$-closed ${\bf C}$-hyperideal of $G_1 \times G_2$ that is not $(s,n)$-closed.
\item[\rm{(ii)}]~ $Q=Q_1 \times Q_2$ for some ${\bf C}$-hyperideals $Q_1$ and $Q_2$ of $G_1$ and $G_2$, respectively, such that either 
\begin{itemize}
\item[\rm{(1)}]~ $Q_1$ is a weakly $(s,n)$-closed hyperideal of $G_1$ that is not $(s,n)$-closed, $0 \in b^s$  whenever $b^s \subseteq Q_2$ for $b \in G_2$. If $0 \notin a^s \subseteq Q_1$ for $a \in G_1$, then $Q_2$ is an $(s,n)$-closed hyperideal of $G_2$, or
\item[\rm{(2)}]~$Q_2$ is a weakly $(s,n)$-closed hyperideal of $G_2$ that is not $(s,n)$-closed, $0 \in b^s$  whenever $b^s \subseteq Q_1$ for $b \in G_1$. If $0 \notin a^s \subseteq Q_2$ for $a \in G_2$, then $Q_1$ is an $(s,n)$-closed hyperideal of $G_1$.
\end{itemize}
\end{itemize}
\end{theorem}
\begin{proof}
(i) $\Longrightarrow$ (ii) Assume that $Q$ is a weakly $(s,n)$-closed ${\bf C}$-hyperideal of $G_1 \times G_2$ that is not $(s,n)$-closed. Then $Q=Q_1 \times Q_2$ for some hyperideals $Q_1$ and $Q_2$ of $G_1$ and $G_2$, respectively, such that one of them is  weakly $(s,n)$-closed  but is not $(s,n)$-closed by Theorem \ref{25}. We may assume that $Q_1$ is a weakly $(s,n)$-closed hyperideal of $G_1$ but is not $(s,n)$-closed. Then we conclude that $Q_1$ has an $(s,n)$-tough-zero element $a_1$ which means $0 \in a_1^s$ and $a_1^n \nsubseteq Q_1$. Let $b^s \subseteq Q_2$ for some $b \in G_2$. Since $(a_1,b)^s \subseteq Q$, we get $(0,0) \in (a_1,b)^s$ which means $0 \in b^s$. Suppose that $0 \notin a^s \subseteq Q_1$ for $a \in G_1$. Let $b \in G_2$ with $b^s \subseteq Q_2$. Therefore $(0,0) \notin (a,b)^s \subseteq Q$. This means that $b^n \subseteq Q_2$ which implies $Q_2$ is an $(s,n)$-closed hyperideal of $G_2$. Similiar for the other case.

(ii) $\Longrightarrow$ (i) Assume that $Q_1$ is a weakly $(s,n)$-closed hyperideal of $G_1$ that is not $(s,n)$-closed. Suppose that $x$ is an $(s,n)$-tough-zero element of $Q_1$. This means that $(x,0)$ is an $(s,n)$-tough-zero element of $Q$. Hence $Q$ is not an $(s,n)$-closed. Let $(0,0) \notin (a,b)^s \subseteq Q$  for  $a \in G_1$ and $b \in G_2$. Since $0 \in b^s$, we get $0 \notin a^s \subseteq Q_1$. Then $(a,b)^n \subseteq Q$ as $Q_1$ is a weakly $(s,n)$-closed hyperideal of $G_1$ and $Q_2$ is an $(s,n)$-closed hyperideal of $G_2$. Consequently, $Q$ is a weakly $(s,n)$-closed ${\bf C}$-hyperideal of $G_1 \times G_2$ that is not $(s,n)$-closed.
\end{proof}



\end{document}